\documentclass{proc-l}

\newtheorem{theorem}{Theorem}[section]
\newtheorem{lemma}[theorem]{Lemma}

\theoremstyle{definition}
\newtheorem{definition}[theorem]{Definition}

\theoremstyle{remark}

\newcommand{\R}{\mathbb{R}}

\newcommand{\dt}{\partial_t}
\newcommand{\dx}{\partial_x}
\newcommand{\eps}{ \varepsilon}
\numberwithin{equation}{section}

\begin{document}

\title[The semi-global isometric embedding]{The semi-global isometric embedding  of surfaces with curvature changing signs stably}


\author[W. Cao]{Wentao Cao}
\address{Institute f\"{u}r mathematik, Universit\"{a}t Leipzig, D-04109, Leipzig, Germany}
\email{wentao.cao@math.uni-leipzig.de}
\thanks{The research  is supported by the ERC Grant Agreement No. 724298. The paper was done when the author was a post-doctor of Max-Plank Institute of Mathematics in the Sciences, so the author warmly thanks the institute for its hospitality and research environment it provided.  The author also thanks the anonymous referee for his/her careful reading and helpful comments, which contributed to an improvement of the initial manuscript.}



\date{}



\begin{abstract}
A semi-global isometric embedding of abstract surfaces with Gaussian curvature changing signs of any finite order is obtained through solving the Darboux equation.
\end{abstract}

\maketitle

%
\section{Introduction}

Isometric embedding is an interesting and historical problem in  differential geometry. Nash obtained two famous results on such problem. In \cite{nash1954}, he showed that any smooth $n$ dimensional Riemannian manifold $(\mathcal{M}^n, g)$ can be $C^1$  isometric embedded into Euclidean space $\R^{n+2}$ by a useful technique called convex integration later, which is also applied to solve non-uniqueness problem in PDE. Besides, in \cite {nash1956} he also proved that  sufficiently smooth metric can be isometric embedded into  Euclidean space $\R^N$ with $N$ depending on $n$ and larger than $s_n=\frac{n(n+1)}{2}$ by applying his implicit theorem, which is a powerful tool handing the loss of regularity.

However, it is believed that the best dimension of the target space for the local smooth isometric embedding of $\mathcal{M}^n$ is $N=s_n.$ Janet in \cite{Janet} and  Cartan in \cite{Cartan} independently proved that any analytic $n$ dimensional Riemannian manifold admits a local analytic isometric embedding in $\R^{s_n}.$ In particular, when $n=2,$ Schlaefli conjectured and later Yau in \cite{Yau} reposed that any smooth surface always admit a local smooth isometric embedding in $\R^3.$ In this case, the problem can be formulated to solve the Gauss-Codazzi system or the Darboux equation (see \cite{HH}). The type of the equations depends on the signs of the Gaussian curvature $K$ of the given surface. When $K>0,$ the equations are of elliptic type. Many mathematicians take the Darboux equation to handle the  problem. It also has a close relation with the Weyl problem when the manifold is $\mathbb{S}^2$ (see \cite{w}), one can find its proof in \cite{n}. Readers can also see \cite{gl, hz,lin0} for $K\geq0$ case. On the other hand, when $K<0,$ or $K\leq0,$ the equations are hyperbolic or degenerate hyperbolic equations. In stead of the Darboux equation, the Gauss-Codazzi system is usually used to tackle the isometric embedding problem for such case, one can see \cite{CSW,Christoforou,H}.

Some difficulties arise when $K$ changes signs, because the equations are of mixed type. Lin made a breakthrough in \cite{Lin} by applying the theory of symmetric positive system and obtained that a sufficiently smooth isometric embedding exists when the Gaussian curvature satisfies
$$K(0)=0, \nabla K(0)\neq0.$$
Then Han improved the regularity of the embedding in \cite{Han} and later he also showed the existence of local isometric embedding when the Gaussian curvature changes signs stably, i.e. $K$ vanishes at finite order across a curve in \cite{han}. For the semi-global isometric embedding, Dong first showed that there exists an isometric embedding parametrized on $[0, 2\pi]\times[-\delta,\delta]$ mapping into $\R^3$ in \cite{Dong} provided that the positive constant $\delta$ is small enough and $K$ satisfies
$$K(x_1, 0)=0, \partial_2K\Gamma^2_{11}<0, \text{ on } [0, 2\pi]\times\{0\}$$
together with the other two  compatibility conditions. Here $(x_1,x_2)$ is the parameter for the surface and $\Gamma^2_{11}$ is one Christoffel symbol. We call an embedding is semi-global if it is periodic in one variable and locally defined in the other variable. He proved this result through the Darboux equations and similar method to \cite{hong1987}.  In \cite{li}, Li got a semi-global isometric embedding under same conditions as Dong but only defined on $[0, 2\pi]\times[0,\delta].$

In the present paper, we gain a semi-global isometric embedding for the case where the Gaussian curvature changes signs stably, i.e. vanishes at any finite order $2\alpha-1$ across a closed curve, by solving the Darboux equation, see Theorem \ref{t:anyorder}. Our main approach is the framework of symmetric positive system, which is fully studied in \cite{Fried} and improved by Gu in \cite{Gu}. Surprisingly,  our priori estimate for the linearised equation does not depend on the vanishing order $\alpha, $ while the regularity of the finial  embedding depends on $\alpha.$ We remark that our arguments can not be extended to  gain the semi-global smooth isometric embedding of the smooth metric, since the $H^\infty$ estimate seems hard to be obtained by the framework here.  Our future plan is to consider the smooth case.

The rest of this paper consists of two sections. In Section \ref{positive}, we provide general theory of symmetric positive system. Main Theorem \ref{t:anyorder} and its proof are given in Section \ref{stable}.

\bigskip

\section{ Symmetric positive system}\label{positive}

The theory of symmetric positive system is given in \cite{Fried} and later improved in \cite{Gu}, which is an effective approach to show the existence of differential equations of mixed type.  We will apply the theory  to obtain the priori estimate of the linearized equation of the Darboux equation  in Section \ref{stable}. In this section, we will give the framework of symmetric positive system, whose details can  be found in  \cite{Fried, Gu, HH}.

Consider the first order linear differential equations
\begin{equation}\label{e:general}
\mathcal{L}U=A\frac{\partial U}{\partial t}+B\frac{\partial U}{\partial x}+CU=F
\end{equation}
defined on a rectangle $\Omega$ centring at 0, where $U:\Omega\mapsto\R^n$ is the unknown vector function and $A, B, C$ are given $n\times n$ real function matrixes, $F:\Omega\mapsto\R^n$ is any given vector function. Besides, the boundary condition for \eqref{e:general} is given as follows
\begin{equation}\label{e:genrealboundary}
\Xi U=0, \quad \text{ on } \partial \Omega,
\end{equation}
where $\Xi$ is  also a given  $n\times n$ real function matrix. If $A$ and $B$ are symmetric, \eqref{e:general} is called symmetric, and \eqref{e:general} is called symmetric positive system if  the following matrix is positive definite,
\begin{equation*}
\Theta=C+C^T-\dt A-\dx B.
\end{equation*}
Here $C^T$ stands for the transpose matrix of $C.$ With the assumption that \eqref{e:general} is symmetric positive, it is easy to get the bound of some norm of $U$ through integrating by parts after multiplying $U$ to \eqref{e:general}. To get a differentiable solution of \eqref{e:general} we must estimate the derivatives of $U.$ An easy calculation gives us
\begin{equation*}
\begin{split}
&A\dt(\dt U)+B\dx(\dt U)+C\dt U+\dt A\dt U+\dt B\dx U+\dt CU=\dt F,\\
&A\dt(\dx U)+B\dx(\dx U)+C\dx U+\dx A\dt U+\dx B\dx U+\dx CU=\dx F.
\end{split}
\end{equation*}
We note that the above equations and \eqref{e:general} still form a symmetric system of equations for the $3n$ unknown functions $\{U, \dt U, \dx U\}.$ However, the symmetric system needs not to be positive and the associated boundary conditions are not necessarily homogeneous. With an observation that the tangential derivatives satisfy the homogeneous boundary conditions, Friedrich in \cite{Fried} introduced a set of differential operators of the first order $D_\sigma$ like
$$D_0=I, D_1=\frac{\partial}{\partial x}, D_\sigma=d^1_\sigma\frac{\partial}{\partial t}+d^2_\sigma, \sigma=2, 3,  $$
where $d^1_\sigma$ are smooth $n\times n$ diagonal matrixes and  $d^2_\sigma$ are any smooth $n\times n$ matrixes. Furthermore, we assume that the  set of differential operator is complete, i.e. any tangential operator $D$ can be expressed as $D=C^\sigma D_\sigma$ with $n\times n$ matrix $C^0$ and scalar functions $C^\sigma(\sigma\neq0).$  On the other hand, let $n_1, n_2$ be the outward normal vectors and $\Upsilon=An_1+Bn_2,$ then the boundary condition \eqref{e:genrealboundary} is said to be admissible if for any point on the boundary $\partial\Omega$, the plane $\Xi U=0$ is the maximal non-negative plane of the quadratic form $U\cdot\Upsilon U$ .
\par
Moreover, the following  formula
\begin{equation}\label{e:diff}
D_\sigma\mathcal{L}=\mathcal{L}D_\sigma+p_\sigma^\tau D_\tau+\theta_\sigma\mathcal{L}
\end{equation}
holds for the set of differential operators. Here $p_\sigma^\tau, \theta_\sigma$ are all $n\times n$ matrices. Let $\underset{1}{U}$ be the set of $4n$ unknown functions $\{U, D_1U, D_2U, D_3U\}$ and define $\underset{1}{\mathcal{L}}$ as the differential operator on $\underset{1}{U}$ like the following
\begin{equation*}
\underset{1}{\mathcal{L}}\underset{1}{U}=\{\mathcal{L}U+p_0^\tau U_\tau, \mathcal{L}U_1+p_1^\tau U_\tau,\cdots,\mathcal{L}U_3+p_3^\tau U_\tau\}
\end{equation*}
with $U_\sigma=D_\sigma U.$ Set
\begin{equation*}
\theta\underset{1}{U}=\{\theta_0U, \theta_1U_1, \theta_2U_2, \theta_3U_3\}.
\end{equation*}
Then we can derive the $1$-st enlarged system
\begin{equation}\label{e:1enlarge}
\underset{1}{\mathcal{L}}\underset{1}{U}=(D-\theta)F
\end{equation}
with notations
\begin{equation*}
\begin{split}
&DF=\{D_0F, D_1F, D_2F, D_3 F\},\\
&\theta F=\{\theta_0F, \theta_1F, \theta_2F, \theta_3 F\}.
\end{split}
\end{equation*}
Indeed for any operator $D_\sigma,$ we have
\begin{equation*}
D_\sigma\mathcal{L}U=\mathcal{L}D_\sigma U+p_\sigma^\tau D_\tau U+\theta_\sigma\mathcal{L}U=\underset{1}{\mathcal{L}}D_\sigma U+\theta_\sigma F=D_\sigma F.
\end{equation*}
On the other hand, if the boundary operator is assumed to satisfy the following formula
\begin{equation}\label{e:diffboundary}
D_\sigma MU=MD_\sigma U+q_\sigma^0D_0U,
\end{equation}
and we furthermore set
$$\underset{1}{\Xi}\underset{1}{U}=\{\Xi U+q_0^0U, \Xi U_1+q_1^0U,\cdots, \Xi U_3+q_3^0U\},$$
then the enlarged boundary condition is
\begin{equation}\label{e:1enlargeboundary}
\underset{1}{\Xi}\underset{1}{U}=0.
\end{equation}
Hence, if $F\in C^1(\bar{\Omega}),$ $U\in C^2(\bar{\Omega}),$ then \eqref{e:1enlarge} and \eqref{e:1enlargeboundary} are satisfied. If the linear differential operator $\mathcal{L}$ and the boundary operator $M$ satisfy \eqref{e:diff} and \eqref{e:diffboundary} respectively, we call that \eqref{e:general}-\eqref{e:genrealboundary} can be enlarged. Friedrich in \cite{Fried} tells us that the set of differential operators can be constructed for any symmetric positive system \eqref{e:general}. Thus, we can also derive any $s$-enlarged system. From \cite{Fried} and \cite{Gu}, a powerful lemma about the existence of differentiable solutions to the boundary value problem \eqref{e:general}-\eqref{e:genrealboundary} can be concluded  as follows.
\begin{lemma}\label{l:exist}
Assume that the system \eqref{e:general} is positive and any $s$-enlarged system of \eqref{e:general} is also positive. Besides, the boundary condition \eqref{e:genrealboundary} is noncharacteristic and admissible. $A, B, C\in C^{s+1},$ $F\in H_s.$ Then there exists a strong solution $U\in H_s.$
\end{lemma}
Here $H_s$ denotes the normed space with norm defined as
$$|||V|||_s=\sum_{0\leq l\leq s}\|D_{\sigma_1}\cdots D_{\sigma_l}V\|_{L^2}^2.$$
We also use $\|\cdot\|_s$ and $|\cdot|_s$ to denote the  norms of Sobolev space $H^s$ and $C^s$ . The proof of the Lemma \ref{l:exist} can be found in \cite{Gu}.
\bigskip

\section{Main theorem and its proof}\label{stable}
Let the given sufficiently smooth Riemaninian metric of a surface be
$$g=g_{ij}dx_idx_j,$$
and consider the isometric embedding problem in the neighbourhood of a closed curve $\Lambda=[0, 2\pi]\times\{0\}$ on the surface, i.e. to seek a surface $\vec{r}$ defined in $I_\delta=[0, 2\pi]\times[-\delta, \delta]$ such that
$$\vec{r}=(p, q, z)(x_1, x_2): I_\delta\rightarrow \R^3,\quad g=d\vec{r}^2.$$
\subsection{Necessary conditions}
Under the geodesic coordinate system based on the curve $\Lambda$, the metric can be reduced to be of the following form
\begin{equation}\label{e:metric}
g=B^2(x_1, x_2)dx_1^2+dx_2^2, \quad B(x_1, 0)=1, \quad \partial_{2}B(x_1, 0)=k_g,
\end{equation}
where $B(x_1, x_2)$ is a sufficiently smooth function and $k_g$ is the geodesic curvature of $\Lambda$. We also denote $(\partial_{x_1}, \partial_{x_2})=(\partial_1, \partial_2)$. Since $\Lambda$ is a closed curve, $B(x_1, x_2)$ is $2\pi$ periodic with respect to $x_1$. In geodesic coordinate, the Christoffel symbols are
\begin{equation}\label{e:christoffel}
\begin{split}
&\Gamma^1_{11}=\frac{\partial_1B}{B}, \quad \Gamma^1_{12}=\frac{\partial_2B}{B}, \quad \Gamma^1_{22}=0,\\
&\Gamma^2_{11}=-B\partial_2B, \quad \Gamma^2_{12}=\Gamma^2_{22}=0.
\end{split}
\end{equation}
As derived in \cite{Dong}, functions $p, q, z$ shall satisfy
\begin{align*}
&dp^2+dq^2=g-dz^2\\
=&(g_{11}-(\partial_1z)^2)dx_1^2+2(g_{12}-\partial_1z\partial_2z)dx_1dx_2+
(g_{22}-(\partial_2z)^2)dx_2^2.
\end{align*}
Then it is not hard  to derive the following Darboux equation satisfied by $z$:
\begin{equation}\label{e:darboux}
\begin{split}
0=&\det(\partial_{ij}z-\Gamma^k_{ij}\partial_kz)-K\det(g_{ij})(1-g^{ij}\partial_iz\partial_jz)\\
=&(\partial_{11}z-\Gamma^1_{11}\partial_1z+\Gamma^1_{12}\partial_2z)\partial_{22}z-(\partial_{12}z-\Gamma^1_{12}\partial_1z)^2\\
&-K\det(g_{ij})(1-g^{ij}\partial_iz\partial_jz)
\end{split}
\end{equation}
with $g^{ij}\partial_iz\partial_jz<1.$ Here matrix $(g^{ij})$ is the inverse of the metric matrix $(g_{ij})$ and we have used \eqref{e:christoffel}.  Finally, the isometric embedding problem is formulated to solve Darboux equation \eqref{e:darboux}. Similar to \cite{Dong} or \cite{li}, we can also derive the following theorem about the necessary conditions for our desired embeddings.
\begin{theorem}\label{t:necessary}
For any sufficiently smooth isometric embedding $\vec{r}=(p, q, z)(x_1, x_2)$ of $g$ on $I_\delta$ into $\R^3$ and $z=O(x_2^{\alpha+1})$ with any integer $\alpha\geq1,$ we have
\begin{align}
&k_g\partial_2^{2\alpha-1}K>0, \text{ on } \Lambda,\label{e:curvature}\\
&\int_0^{2\pi}k_gdx_1=2\pi, \label{e:compati-1}\\
&\int_0^{2\pi}\exp{\sqrt{-1}}\int_0^{x_1}k_gdsdx_1=0. \label{e:compati-2}
\end{align}
\end{theorem}
\begin{proof}
Since $\vec{r}$ is an isometric embedding of $I_\delta$, and $z=O(x_2^{\alpha+1})$, we directly take $2\alpha-1$ times partial derivatives of  \eqref{e:darboux} with respect to $x_2$ and then set $x_2=0$, finally obtain when $x_2=0,$
\begin{equation*}
-\Gamma^2_{11}(\partial_2^{\alpha+1}z)^2=\partial_2^{2\alpha-1}K\det{g_{ij}}\neq0,
\end{equation*}
which implies
 $$\Gamma^2_{11}(x_1, 0)\neq0, \quad \partial_2^{\alpha+1}z\neq0,\quad(\Gamma^2_{11}\partial_2^{2\alpha-1}K)(x_1, 0)<0.$$
Hence, taking value of \eqref{e:metric} and \eqref{e:christoffel}, one can easily derive \eqref{e:curvature}. Using the fact that $\Lambda$ is a closed curve and following the same procedure as \cite{Dong}, we can also gain \eqref{e:compati-1} and \eqref{e:compati-2}.
\end{proof}

\subsection{Statement of main theorem}
We first give a definition.
\begin{definition}
A surface is called $\alpha$-surface if it is  $2\pi$ periodic with respect to $x_1$ and satisfies \eqref{e:curvature}-\eqref{e:compati-2}.
\end{definition}
Then we state our semi-global isometric embedding theorem.
\begin{theorem}[\bf{Main Theorem}]\label{t:anyorder}
Given an $\alpha$-surface prescribed with sufficiently smooth ($C^{s_*}, s_*\geq 2\alpha+31\}$)    metric $g$, we can find a small positive constant $\delta$ and a sufficiently smooth ($C^{s}, 4\leq s\leq \tfrac{4}{7} (s^*-2\alpha)-4$) isometric embedding
$$\vec{r}=(p, q, z)(x_1, x_2): I_\delta\rightarrow \R^3,$$
such that $g=d\vec{r}^2.$
\end{theorem}

\subsection{Proof of main theorem}
We divide the proof into four steps.

\emph{Step 1. Initial approximate solution. } Without loss of generality, we can assume that  the Gaussian curvature is
$$K(x_1, x_2)=x_2^{2\alpha-1}K_0(x_1, x_2), \quad K_0(x_1, x_2)\neq0, \text{ on } \Lambda.$$
Then \begin{equation}\label{e:k0}
k_gK_0(x_1, 0)>0.
\end{equation}
Let our desired solution to the Darboux equation be of the following form
$$z(x_1, x_2)=x_2^{\alpha+1}(a(x_1)+w(x_1, x_2))$$
and  then an easy calculation yields
\begin{equation}\label{e:zderivatives}
\begin{split}
&\partial_1z=x_2^{\alpha+1}(a'+\partial_1w), \quad \partial_2z=x_2^\alpha(\alpha+1)(a+w)+x_2^{\alpha+1}\partial_2w,\\
&\partial_{11}z=x_2^{\alpha+1}(a''+\partial_{11}w), ~~~\partial_{12}z=x_2^\alpha(\alpha+1)(a'+\partial_1w)+x_2^{\alpha+1}\partial_{12}w,\\
&\partial_{22}z=(\alpha+1)\alpha x_2^{\alpha-1}(a+w)+2(\alpha+1)x_2^\alpha\partial_2w+x_2^{\alpha+1}\partial_{22}w.
\end{split}
\end{equation}
Plugging \eqref{e:zderivatives} into \eqref{e:darboux} and then dividing the resulting equation by $x_2^{2\alpha}$, we will get
\begin{equation}\label{e:initial-1}
\begin{split}   
&[x_2\partial_{11}w+x_2a''-\Gamma_{11}^1x_2(a'+\partial_1w)-\Gamma^2_{11}((\alpha+1)(a+w)+x_2\partial_2w)] \\
&\cdot[x_2\partial_{22}w+2(\alpha+1)\partial_2w+(\alpha+1)\alpha(a+w)x_2^{-1}]\\
&-[x_2\partial_{12}w+(\alpha+1)(a'+\partial_1w)-\Gamma^1_{12}x_2(a'+\partial_1w)]^2 \\
&-x_2^{-1}K_0\det(g_{ij})\big[1-g^{11}(x_2^{\alpha+1}(a'+\partial_1w))^2\\
&\qquad-g^{22}(x_2^\alpha(\alpha+1)(a+w)+x_2^{\alpha+1}\partial_2w)^2\\
&\qquad-2g^{12}(x_2^\alpha(\alpha+1)(a+w)+x_2^{\alpha+1}\partial_2w)(x_2^{\alpha+1}(a'+\partial_1w))\big]=0.
\end{split}
\end{equation}
To construct $a(x_1),$ let $w=0$ in \eqref{e:initial-1} and then we get
\begin{equation}\label{e:initial-2}
\begin{split}
&[-\Gamma_{11}^1x_2a'-\Gamma^2_{11}(\alpha+1)a]
\cdot[(\alpha+1)\alpha ax_2^{-1}]-[(\alpha+1)a'-\Gamma^1_{12}x_2a']^2-\\
&-x_2^{-1}K_0\det(g_{ij})\big[1-g^{11}(x_2^{\alpha+1}a')^2
g^{22}(x_2^\alpha(\alpha+1)a)^2-2g^{12}ax_2^{2\alpha+1}a'\big]=0
\end{split}
\end{equation}
Furthermore, taking $x_2=0$  in \eqref{e:initial-2} leads to that
$$-\Gamma^2_{11}(\alpha+1)^2\alpha a^2x_2^{-1}-x_2^{-1}K_0\det(g_{ij})=0$$
holds for $x_2=0.$ Thus we take
\begin{equation}\label{e:a}
a(x_1)=\sqrt{\frac{-K_0\det(g_{ij})}{\Gamma^2_{11}\alpha(\alpha+1)^2}}(x_1, 0)
=\sqrt{\frac{K_0\det(g_{ij})}{k_g\alpha(\alpha+1)^2}}(x_1, 0).
\end{equation}
It is easy to see that $a(x_1)$ is well defined from \eqref{e:k0}. After choosing $a(x_1),$  we take
$$x_1=y_1, x_2=\eps^2 y_2, w=\eps u.$$
Then we can rewrite the Darboux equation as follows.
\begin{equation}\label{e:initial-3}
\begin{split}   
\mathcal{F}(u, \eps)=
&[\eps^3y_2\partial_{11}u+\eps^2y_2a''-\Gamma_{11}^1\eps^2y_2(a'+\eps\partial_1u)\\
&\quad -\Gamma^2_{11}((\alpha+1)(a+\eps u)+\eps y_2\partial_2u)] \\
&\cdot[y_2\partial_{22}w+2(\alpha+1)\partial_2u+(\alpha+1)\alpha(a+\eps u)(\eps y_2)^{-1}\\
&-\eps[\eps y_2\partial_{12}u+(\alpha+1)(a'+\eps\partial_1u)-\Gamma^1_{12}\eps^2y_2(a'+\eps\partial_1u)]^2 \\
&-{\eps y_2}^{-1}K_0\det(g_{ij})\big[1-g^{11}((\eps^2y_2)^{\alpha+1}(a'+\eps\partial_1u))^2\\
&\quad-g^{22}((\eps^2y_2)^\alpha(\alpha+1)(a+\eps u)+(\eps^2y_2)^{\alpha+1}\eps^{-1}\partial_2u)^2\\
&\quad-2g^{12}(\eps^2y_2)^{2\alpha+1}((\alpha+1)(a+\eps u)+\eps y_2\partial_2u)(a'+\eps\partial_1u)\big]=0.
\end{split}
\end{equation}
Here we still use $\partial_i$ to denote taking derivatives with respect to $y_i.$
Utilizing the definition of $a(x_1),$ it is easy to find
 $$\mathcal{F}(0, \eps)=\eps F_0(y_1, \eps^2y_2),$$
where $F_0(x_1,x_2)$ is a smooth function of $x_1, x_2.$ Note that $a(x_1)\in C^{s_*-2\alpha-1} $ and $K_0(x_1, x_2)\in C^{s_*-2\alpha-1} $ due to  $g\in C^{s_*}$ and $K\in C^{s_*-2}.$ Hence
$F_0(y_1, \eps^2y_2)\in C^{s_*-2\alpha-3}$ since $a''(y_1)$ in  $F_0(y_1, \eps^2y_2)$ is of $C^{s_*-2\alpha-3}.$

\emph{Step 2. Linearisation.} Linearising of $\mathcal{F}$ at $u$ contributes to
\begin{equation*}
\mathcal{F}'(u)\phi=\underset{i, j=1,2}{\sum}a_{ij}\partial_{ij}\phi+\sum_{i=1,2}a_i\partial_i\phi+a_0\phi,
\end{equation*}
in which
\begin{align*}
a_{11}=&\eps^3y_2[y_2\partial_{22}u+2(\alpha+1)\partial_2u+(\alpha+1)\alpha(a+\eps u)(\eps y_2)^{-1}]\\
\doteq&\eps^2[a(\alpha+1)\alpha+\eps \tilde{b}_{11}],\\
a_{12}=&-\eps y_2[\eps^2y_2\partial_{12}u+\Gamma^1_{11}\eps^3y_2(a'+\eps\partial_1u)+\eps(\alpha+1)a'\\
&+\eps\Gamma^2_{11}((\alpha+1)(a+\eps u)+\eps y_2\partial_2u)+\eps^2(\alpha+1)\partial_1u]\\
\doteq&\eps^2y_2\tilde{b}_{12},\\
a_{22}=&y_2[\eps^3y_2\partial_{11}u+\Gamma^1_{11}\eps^2y_2(a'+\eps\partial_1u)+\Gamma^2_{11}(\alpha+1)(a+\eps u)
+\Gamma^2_{11}\eps y_2\partial_2u]\\
\doteq&y_2[\Gamma^2_{11}(\alpha+1)a+\eps \tilde{b}_{22}],
\end{align*}
and
\begin{align*}
a_1=&\eps^3y_2\Gamma_{11}^1\frac{a_{11}}{\eps^3y_2}
+[\eps^2(\alpha+1)+\Gamma^1_{11}\eps^4y_2]\frac{a_{12}}{\eps y_2}+h.o.t.,\\
\doteq&\eps^2\tilde{b}_1,\\
a_2=&\eps\Gamma^2_{11}y_2\frac{a_{11}}{\eps^3y_2}+2(\alpha+1)\frac{a_{22}}{y_2}
+h.o.t,\\
\doteq&\Gamma_{11}^2(\alpha+1)(3\alpha+2)a+\eps \tilde{b}_2,\\
a_0=&\eps(\alpha+1)\Gamma^2_{11}\frac{a_{11}}{\eps^3y_2}+\frac{(\alpha+1)\alpha a}{\eps y_2}\frac{a_{22}}{y_2}\\
\doteq&\frac{1}{y_2}[2\Gamma^2_{11}a(\alpha+1)^2\alpha+\eps \tilde{b}_0] ,
\end{align*}
where $h.o.t.$ stands for the higher order terms with respect to $\eps$.  Besides,  we
divide $\mathcal{F}(u, \eps)\phi$ by $\Gamma^2_{11}a(\alpha+1)$ and take a new change of variable
\begin{equation*}
t=y_2, x=\frac{2\pi}{\int^{2\pi}_0\sqrt{-(\alpha+1)k_ga}dx_1}
\int_0^{y_1}\sqrt{-(\alpha+1)k_ga}dx_1,
\end{equation*}
where we have assumed $\Gamma^2_{11}(x_1, 0)=-k_g>0$ (otherwise one can divide $\mathcal{F}'(u)\phi$ by $-\Gamma^2_{11}a(\alpha+1),$ and make similar transformation).
We finally get the following linear differential equation
\begin{equation*}
\begin{split}
\mathcal{L}(u)\phi=&\eps^2(\alpha+\eps b_{11})\partial^2_{xx}\phi+\eps^2tb_{12}\partial^2_{xt}\phi+t(1+\eps b_{22})\partial^2_{tt}\phi\\
&+\eps^2b_1\dx\phi+(3\alpha+2+\eps b_2)\dt\phi+[2(\alpha^2+\alpha)+\eps b_0]\frac{\phi}{t},
\end{split}
\end{equation*}
where $b_{ij}, b_k, (i, j=1,2. k=0,1,2)$ are all bounded smooth linear functions with respect to $u,  t\partial_tu, t^2\partial^2_{tt}u, \partial_xu,\partial^2_{xx}u, t\partial_{xt}^2u$, similar to $\tilde{b}_{ij}, \tilde{b}_k.$

\emph{Step 3. Priori estimates.}  For any given smooth function $f(x, t),$ we  study the boundary value problem
\begin{equation*}
\begin{split}
&\mathcal{L}(u)\phi=f(x, t), \quad (x, t)\in G=[0, 2\pi]\times[-2,2],\\
&\phi(x, 2)=0, x\in[0, 2\pi]; \phi(0, t)=\phi(2\pi, t),  |t|\leq2,
\end{split}
\end{equation*}
and derive the estimates of its solutions.
Let
$$U=(u_1, u_2, u_3)^T=e^{\gamma t}(\dt\phi, \frac{\phi}{t}, \eps\dx\phi),$$
and $F=(e^{\gamma t}f, 0, 0)^T.$
Then the linear equation with boundary condition can be transformed to be the following boundary value problem
\begin{equation}\label{e:linear}
\begin{split}
&Lu=A\dt U+B\dx U+CU=F,\\
&u_2(x, 2)=u_3(x, 2)=0, x\in[0, 2\pi], \\
&U(0, t)=U(2\pi, t), |t|\leq 2,
\end{split}
\end{equation}
where
\begin{equation*}
\begin{split}
&A=\left(               
  \begin{array}{ccc}   
   t(1+\eps b_{22})~&~0~&~0\\  
   0~&~\beta t~&~0\\
   0~&~0~&~-(\alpha+\eps b_{11})  
  \end{array}
\right),\\
&B=\left(               
  \begin{array}{ccc}   
   \eps^2b_{12}t~&~0~&~\eps(\alpha+\eps b_{11})\\  
   0~&~0~&~0\\
   \eps(\alpha+\eps b_{11})~&~0~&~0 
  \end{array}
\right),\\
&C=\left(               
  \begin{array}{ccc}   
   3\alpha+2-\gamma t+\eps b_2-\eps\gamma t b_{22}~&~
   2(\alpha^2+\alpha)+\eps b_0~&~\eps b_1\\  
   -\beta~&~\beta(1-\gamma t)~&~0\\
   0~&~0~&~\gamma(\alpha+\eps b_{11}) 
  \end{array}
\right)
\end{split}
\end{equation*}
with positive constants  $\beta, \gamma$ to be determined.
In our transformation, we have used the following two simple equations:
\begin{equation*}
\begin{split}
\beta t\dt(e^{\gamma t}\frac{\phi}{t})&=\beta t(\gamma e^{\gamma t}\frac{\phi}{t}+e^{\gamma t}\frac{\dt\phi}{t}-e^{\gamma t}\frac{\phi}{t^2})\\
&=\beta e^{\gamma t}\dt\phi+\beta(\gamma t-1)e^{\gamma t}\frac{\phi}{t},\\
\dt(e^{\gamma t}\eps\dx\phi)&=\gamma(e^{\gamma t}\dx\phi)+\eps\dx(e^{\gamma t}\dt\phi).
\end{split}
\end{equation*}
We shall show any s-enlarged system of \eqref{e:linear} is symmetric positive so that we can use Lemma \ref{l:exist} to show the existence and derive estimates. It is easy to see that $Lu=F$ is a symmetric system. To prove it positive, let $\eps=0,$ then \eqref{e:linear} becomes
$$L^0U=A^0\dt U+C^0U=F,$$
where
\begin{equation*}
\begin{split}
&A^0=\left(               
  \begin{array}{ccc}   
   t~&~0~&~0\\  
   0~&~\beta t~&~0\\
   0~&~0~&~-\alpha  
  \end{array}
\right),\\
&C^0=\left(               
  \begin{array}{ccc}   
   3\alpha+2-\gamma t~&~2(\alpha^2+\alpha)~&~0\\  
   -\beta~&~\beta(1-\gamma t)~&~0\\
   0~&~0~&~\gamma\alpha 
  \end{array}
\right).
\end{split}
\end{equation*}
Furthermore, simple calculation leads to
$$
\underset{0}{\Theta}^0=C^0+C^{0T}-\dt A^0
=\left(               
  \begin{array}{ccc}   
   2(3\alpha+2-\gamma t)-1~&~
   2(\alpha^2+\alpha)-\beta~&~0\\  
   2(\alpha^2+\alpha)-\beta~&~2\beta(1-\gamma t)-\beta~&~0\\
   0~&~0~&~2\gamma\alpha 
  \end{array}
\right).
$$
Hence we can take
$$\beta=2(\alpha^2+\alpha), \quad 0<\gamma<\frac{1}{4}(\text{e.g.} \gamma=\frac{1}{8})$$
to make $\underset{0}{\Theta}^0$ positive. Assume that $|u|_4\leq1,$  we can choose $\eps_1$ small enough so that when $0<\eps\leq \eps_1,$ $$\underset{0}{\Theta}=C+C^T-\dt A-\dx B$$
is positive. On the other hand, similar to \cite{Dong}, we introduce the following differential operators
$$\mathfrak{D}=\{D_0=I, D_1=\dx, D_2=\xi_2(t)\dt, D_3=\xi_3(t)(t-2)\dt\},$$
which forms a complete system of tangential differential operators on $G$ if $\xi_2+\xi_3=1$ on $G$ and $\xi_2=1$ when $t<1/2,$ $\xi_3=1$ when $t>1.$ It is not hard to get that $\det A(x, 2)\neq0$ when $x\in[0, 2\pi]$ and $0<\eps\leq\eps_2$ with $\eps_2$ small enough, thus $t=2$ is not characteristic boundary for the system \eqref{e:linear}. Moreover, the $s$-th enlarged system of \eqref{e:linear} is
\begin{equation}\label{senlarge2}
\begin{split}
LD_{\sigma_1}\cdots D_{\sigma_s}&=\sum_\mu P_{\sigma_a}^\mu D_{\sigma_1}\cdots D_{\sigma_{a-1}}D_\mu D_{\sigma_{a+1}}\cdots D_{\sigma_s}+\prod^s_{i=1}(D_{\sigma_i}-Q_{\sigma_i})L\\
&+\sum_{r\leq s-1}D^{q_1}Q_{\sigma_1}\cdots D^{q_{l-1}}Q_{\sigma_{l-1}}D^{q_l}P_{\sigma_l}^\mu D_\mu  D_{\sigma_{l+1}}\cdots D_{\sigma_r},
\end{split}
\end{equation}
with integers satisfying $q_1+\cdots+q_l\geq1,$ $q_1+\cdots+q_l+1+r-l\geq s.$ Here $P^\mu_\sigma, Q$ are smooth matrices. The positivity of \eqref{senlarge2} is determined by
$$\underset{s}{\Theta}=\text{diag}(\underset{0}{\Theta}^0,\cdots, \underset{0}{\Theta}^0)+
\text{diag}(m_1\dt A, \cdots, m_{r_0}\dt A)$$
with $r_0=3\cdot 4^s$ and integers $m_j\in[0, s], 1\leq j\leq r_0.$ Hence, $\underset{s}{\Theta}$ is positive definite provided that $\eps$ is small enough and $|u|_4\leq1$.
Upon obtaining the positivity of any $s$-th enlarged system, Lemma \ref{l:exist} guarantees existence of \eqref{e:linear}. Taking value of the relation  between $H_s$ and $H^s$ and Sobolev embedding theorem, following\cite{Dong}, one can obtain the priori estimate for the linearised boundary value problem
\begin{equation}\label{e:estimate}
\|U\|_s\leq C_s(\|F\|_s+\zeta(s)\|u\|_{s+3}\|F\|_2),
\end{equation}
with $\zeta(s)=1$ when $s\geq3,$ and $\zeta(s)=0$ when $0\leq s\leq2.$

\emph{Step 4. Iteration and seeking $\vec{r}$.} After comparing \eqref{e:estimate} with the priori estimates (34) in \cite{Dong}, one find that both priori estimates are same. Hence  we can use Nash-Moser iteration scheme in \cite{Dong} to construct a sequence of approximate solutions.  After taking $0<\eps\leq\eps_3$ with $\eps_3$ small, we are able to show its convergence to a function $U\in H^{\tilde{s}}$ such that
$$\|U\|_{\tilde{s}}\leq C\eps\|F_0\|_{\tilde{s}_*} $$
with
$$ 6\leq\tilde{s}<\frac{4}{7}\tilde{s}_*, \quad 28\leq\tilde{s}_*=s_*-2\alpha-3. $$
Hence our desired solution $z\in C^s$ with $s=\tilde{s}-2$ and then
$$4\leq s\leq\frac{4}{7}(s_*-2\alpha)-4,\quad s_*\geq2\alpha+31.$$
Following the same way as Section 4 of \cite{Dong}, we are able to seek $p, q$, where the two compatibility conditions \eqref{e:compati-1}-\eqref{e:compati-2} are used to make $p, q, z$ periodic in $x_1$. Therefore we complete the proof of Theorem \ref{t:anyorder} after taking $\delta=2\eps_0$ with $\eps_0=\min\{\eps_i, i=1, 2, 3\}.$

%
\bigskip



\bibliographystyle{amsplain}


\end{document}